\documentclass{amsart}

\usepackage{amsmath}

\usepackage{cite}
\usepackage{amsthm,thmtools}
\usepackage{amssymb}
\usepackage{ytableau}
\usepackage[mathscr]{eucal}
\usepackage{diagrams}
\usepackage{tikz-cd}
\usepackage{hyperref}
\usepackage{enumitem}
\usepackage{microtype}
\usepackage{tabularx}
\usepackage{array}

\let\AA\relax
\newcommand{\AA}{\mathbb{A}}

\newcommand{\NN}{\mathbb{N}}
\let\SS\relax
\newcommand{\SS}{\mathbb{S}}
\newcommand{\QQ}{\mathbb{Q}}

\newcommand{\PP}{\mathbb{P}}
\newcommand{\ZZ}{\mathbb{Z}}
\newcommand{\ZZZ}{\widehat{\ZZ}}

\newcommand{\F}{\mathcal{F}}

\newcommand{\op}{\mathrm{op}}

\let\P\relax
\newcommand{\P}{\mathrm{P}}

\let\M\relax
\DeclareMathOperator{\M}{M}

\DeclareMathOperator{\Hom}{Hom}

\DeclareMathOperator{\GL}{GL}
\DeclareMathOperator{\PGL}{PGL}
\DeclareMathOperator{\PSL}{PSL}
\DeclareMathOperator{\SL}{SL}
\DeclareMathOperator{\aut}{Aut}

\newcommand{\sets}{\mathsf{Set}}
\newcommand{\setswith}[1]{#1\text{-}\sets}

\newcommand{\sh}{\mathsf{Sh}}
\newcommand{\psh}{\mathsf{PSh}}

\newcommand{\nat}{\mathrm{Nat}}
\newcommand{\inn}{\mathrm{Inn}}
\newcommand{\out}{\mathrm{Out}}
\newcommand{\ext}{\mathrm{Ext}}

\newcommand{\fP}{\mathfrak{P}}

\newcommand{\fM}{\mathfrak{M}}

\newcommand{\ns}{\mathrm{ns}}

\let\max\relax
\newcommand{\max}{\mathrm{max}}

\newtheorem{definition}{Definition} %[section]
\newtheorem{proposition}[definition]{Proposition}

\newtheorem{theorem}[definition]{Theorem}
\newtheorem{corollary}[definition]{Corollary}
\newtheorem{lemma}[definition]{Lemma}

\newtheorem{remark}[definition]{Remark}

\tikzset{
b/.style={bend left=10},
bb/.style={bend left},
cl/.style={outer sep=-1pt},
}

\let\phi\varphi
\let\theta\vartheta

\newcommand{\pctext}[2]{\text{\parbox{#1}{\centering #2}}}

\title{An arithmetic topos for integer matrices}
\author{Jens Hemelaer}
\thanks{The author is a Ph.D.\ fellow of the Research Foundation -- Flanders (FWO)}

\begin{document}

\begin{abstract}
We study the topos of sets equipped with an action of the monoid of regular $2 \times 2$ matrices over the integers. In particular, we show that the topos-theoretic points are given by the double quotient $\GL_2(\ZZZ) \backslash \M_2(\AA_f) \slash \GL_2(\QQ)$, so they classify the groups $\ZZ^2 \subseteq A \subseteq \QQ^2$ up to isomorphism. We determine the topos automorphisms and then point out the relation with Conway's big picture and the work of Connes and Consani on the Arithmetic Site. As an application to number theory, we show that classifying extensions of $\QQ$ by $\ZZ$ up to isomorphism relates to Goormaghtigh conjecture.
\end{abstract}

\maketitle
\tableofcontents

\section{Introduction}

In \cite{ConnesConsani}, Connes and Consani introduced the Arithmetic Site: it is the topos of sets with an action of the monoid $\NN_+^\times$ of strictly positive natural numbers under multiplication. It comes equipped with a structure sheaf given by a tropical semiring with $\NN_+^\times$-action. While the structure sheaf is necessary in order to provide a link with the Riemann Hypothesis, the topos of $\NN_+^\times$-sets itself is already surprisingly interesting. For example, by \cite{llb-covers} the points of the topos are given by $\QQ^\times\backslash\AA_f\slash\ZZZ^\times$, with $\ZZZ = \prod_p \ZZ_p$ the profinite integers, $\ZZZ^\times$ its units and $\AA_f = \ZZZ\otimes\QQ$ the finite adeles. When taking the structure sheaf in account, the points of the topos become $\QQ^\times\backslash\AA\slash\ZZZ^\times$, with $\AA = \AA_f \times \mathbb{R}$ the full ring of adeles, by the results of \cite{ConnesConsani}. In this way, the Arithmetic Site provides a geometric meaning to both $\QQ^\times\backslash\AA_f\slash\ZZZ^\times$ and $\QQ^\times\backslash\AA\slash\ZZZ^\times$.

We will generalize the construction of the Connes--Consani arithmetic site by considering the topos $\setswith{\M_2^\ns(\ZZ)}$ of sets equipped with a left $\M_2^\ns(\ZZ)$-action, where $\M^\ns_2(\ZZ)$ denotes the $2\times2$-matrices with integer coefficients and nonzero determinant (as a monoid under multiplication). First we show that $\setswith{\M_2^\ns(\ZZ)}$ is covered as a topos by an easy to describe topological space. This allows us to show that the topos points of $\setswith{\M_2^\ns(\ZZ)}$ are given by
\begin{equation} \label{eq:double-quotient}
\GL_2(\ZZZ)\backslash\M_2(\AA_f)\slash\GL_2(\QQ)
\end{equation}
(note the similarity to the case of the Arithmetic Site). It turns out that this double quotient also classifies the abelian groups $\ZZ^2 \subseteq A \subseteq \QQ^2$ up to isomorphism. This gives an alternative to the similar classification of these groups up to isomorphism by Mal'cev in \cite{Malcev}. In \mbox{Section \ref{sec:applications}}, we study to what extent the double quotient (\ref{eq:double-quotient}) leans itself to calculations. We provide an alternative proof for the isomorphism $\ext^1(\QQ,\ZZ) \cong \AA_f/\QQ$ and then give an adelic criterion for when two extensions are isomorphic (as abelian groups).

Note that the $\NN_+^\times$ is the free commutative monoid with the prime numbers as generators. In particular, the prime numbers are indistinguishable from each other: for each permutations of the prime numbers, there is an induced automorphism of $\NN_+^\times$. This in turn induces a topos automorphism of $\setswith{\NN_+^\times}$ (the topos of $\NN_+^\times$-sets and equivariant maps). An important implication is that the topos $\setswith{\NN_+^\times}$ contains no information at all about the Riemann Hypothesis. This is one of the reasons why the tropical semiring (as structure sheaf) is so important in the approach of Connes and Consani. In Subsection \ref{ssec:automorphisms} we compute the topos automorphisms of $\setswith{\M_2^\ns(\ZZ)}$. This topos is much more rigid --- in particular, each automorphism acts trivially on the space of points. So in some sense, the topos $\setswith{\M_2^\ns(\ZZ)}$ contains more arithmetic information than $\setswith{\NN_+^\times}$; and maybe the information in $\setswith{\M_2^\ns(\ZZ)}$ can (partially) replace the role played by the structure sheaf in \cite{ConnesConsani}, leading to an even more ``algebraic'' approach.

An alternative take on this can be found in Subsection \ref{ssec:pz}, where the monoid $\bar{\P}^\ns(\ZZ)$ is studied. It is the submonoid of $\bar{\P}(\ZZ)$ consisting of the matrices with nonzero determinant, where
\begin{equation}
\bar{\P}(R) = \left\{ \begin{pmatrix}
a & b \\
0 & 1
\end{pmatrix} ~:~ a,b \in R \right\}.
\end{equation}
for $R$ a commutative ring. These sets of matrices are used by Connes and Consani in \cite{ConnesConsaniLifts} to study parabolic $\QQ$-lattices. We will show that the topos points of $\setswith{\bar{\P}^\ns(\ZZ)}$ agree with the points of the Arithmetic Site (if we do not take into account the structure sheaf). Moreover, the zeta function naturally associated to $\bar{\P}^\ns(\ZZ)$ is the Riemann zeta function $\zeta(s)$, and the group of topos automorphisms is isomorphic to $\ZZ/2\ZZ$.

In Section \ref{sec:conway}, we discuss the relationship of $\setswith{\M_2^\ns(\ZZ)}$ with Conway's big picture (as introduced in \cite{Conway}). We consider an embedding of the big picture $\mathfrak{P}$ in the quotient
\[
\GL_2(\ZZ)\backslash\M_2^\ns(\ZZ)
\]
(this embedding already appeared in \cite{Plazas}). We give an explicit formula for the hyper-distance on $\GL_2(\ZZ)\backslash\M_2^\ns(\ZZ)$, extending the hyper-distance on the big picture. Then we show that the zeta function associated to the big picture is
\begin{equation}
\zeta_{\fP}(s) = \frac{\zeta(s)\zeta(s-1)}{\zeta(2s)}
\end{equation}
with $\zeta(s)$ the Riemann zeta function. Note that the Riemann zeta function associated to $\GL_2(\ZZ)\backslash\M_2^\ns(\ZZ)$ is $\zeta(s)\zeta(s-1)$. The latter is a special case of a result from \cite{Saito}, but it is also implicit in the work of Connes and Marcolli \cite{ConnesMarcolli}, who showed that $\zeta(s)\zeta(s-1)$ is the partition function for their $\GL_2$-system. Note that $\zeta(s)\zeta(s-1)$ is the Hasse-Weil zeta function for $\PP_\ZZ^1$. This hints at an interpretation of $\setswith{\M_2^\ns(\ZZ)}$ in terms of algebraic geometry.

Another link with the work of Connes and Marcolli \cite{ConnesMarcolli} is that their $\GL_2$-system $\mathcal{A}$ can be interpreted as an algebra of operators on a vector space internal to the topos $\setswith{\M_2^\ns(\ZZ)}$. In more down-to-earth terms: there is some vector space $E$ equipped with a linear left $\M_2^\ns(\ZZ)$-action, such that the $\GL_2$-system acts faithfully on $E$ in an equivariant way. Indeed, take
\begin{equation}
E = \bigoplus_{y} \ell^2(\Gamma \backslash G_y)
\end{equation}
with $y = (\rho,\tau) \in \M_2(\ZZZ)\times\mathfrak{H}$, for $\mathfrak{H}$ the upper half-plane, and \[G_y = \{ g \in \GL_2^+(\QQ) ~:~ g\rho \in \M_2(\ZZZ) \}\]
and let $a \in \mathcal{A}$ act as $\pi_y(a)$ on $\ell^2(\Gamma\backslash G_y)$, where $\pi_y$ is the usual representation as constructed by Connes and Marcolli \cite[Proposition 1.23]{ConnesMarcolli}. Then $\mathcal{A}$ acts in an equivariant way, provided we equip $E$ with the following left $\M_2^\ns(\ZZ)$-action: if $\xi \in \ell^2(\Gamma\backslash G_y)$, then $a \cdot \xi \in \ell^2(\Gamma \backslash G_{y'})$ with $y' = (a \cdot \rho, a \cdot \tau)$ and moreover
\begin{equation}
(a \cdot \xi)(g) = \xi(ga)
\end{equation}
for all $g \in G_{y'}$. Note that $a$ can have negative determinant, but this issue is resolved by considering the identification
\begin{equation}
\GL_2(\ZZ)\backslash\M_2^\ns(\ZZ) = \SL_2(\ZZ)\backslash\M_2^+(\ZZ).
\end{equation}

\section{The topos of \texorpdfstring{$\M_2^\ns(\ZZ)$}{M2ns(Z)}-sets} \label{sec:topos}

Let $\M_2^\ns(\ZZ) = \{ a \in \M_2(\ZZ) : \det(a) \neq 0 \}$ be the regular integral $2 \times 2$-matrices, considered as a monoid under multiplication. We want to construct a topological space $X$ with a surjective open geometric morphism
\begin{equation}
\sh(X) \longrightarrow \setswith{\M_2^\ns(\ZZ)}.
\end{equation}
This map should additionally induce a surjection on topos-theoretic points, leading to a description of the topos-theoretic points for $\setswith{\M_2^\ns(\ZZ)}$.

Two topological spaces for which we can construct a geometric morphism as above (see Subsection \ref{ssec:points}) are
\[
\GL_2(\ZZ) \backslash \M_2^\ns(\ZZ) ~\text{ and }~ \GL_2(\ZZZ) \backslash \M_2(\ZZZ)
\]
with the Scott topology. We will show that the second one is the sobrification of the first one. So they share the same topos of sheaves and the space of points for this topos is  $\GL_2(\ZZZ) \backslash \M_2(\ZZZ)$. We will give a concrete description of both these spaces and relate them to subgroups of $\QQ^2$.

After constructing a covering of toposes
\begin{equation}
\sh\left(\GL_2(\ZZ)\backslash\M_2^\ns(\ZZ)\right) \longrightarrow \setswith{\M_2^\ns(\ZZ)}
\end{equation}
we show that two points of the first topos, seen as elements $a,b \in \GL_2(\ZZZ)\backslash\M_2(\ZZZ)$, give rise to the same topos-theoretic point of $\setswith{\M_2^\ns(\ZZ)}$ if and only if they are $g, h \in \M_2^\ns(\ZZ)$ such that
\begin{equation}
a \cdot g = b \cdot h.
\end{equation}
In this way we can identify the space of points of $\setswith{\M_2^\ns(\ZZ)}$ with
\begin{equation}
\GL_2(\ZZ) \backslash \M_2(\AA_f) \slash \GL_2(\QQ).
\end{equation}

\subsection{Posets of matrices} \label{ssec:posets}

In this subsection we give a concrete description of the posets $\GL_2(\ZZ)\backslash\M_2^\ns(\ZZ)$ and $\GL_2(\ZZZ)\backslash\M_2(\ZZZ)$. So it can be skipped if the reader already has an intuition about these posets and/or has no interest in the proofs of the next subsection.

We already mentioned the topological spaces
\[
\GL_2(\ZZ) \backslash \M_2^\ns(\ZZ) ~\text{ and }~ \GL_2(\ZZZ)\backslash \M_2(\ZZZ)
\]
with the Scott topology. The Scott topology is only defined for posets, so we should clarify which partial order we have in mind. For $a,b \in \M_2^\ns(\ZZ)$ we say $a \leq b$ if there is some $m \in \M_2^\ns(\ZZ)$ with $b = ma$. Clearly, $m = ba^{-1}$ is uniquely determined if it exists. Moreover, $a \leq b$ and $b \leq a$ implies that $b = ua$ for some $u \in \GL_2(\ZZ)$. So we get a partial order on the quotient $\GL_2(\ZZ)\backslash\M_2^\ns(\ZZ)$. The definition for $a,b \in \GL_2(\ZZZ)\backslash\M_2(\ZZZ)$ is completely analogous: $a \leq b$ if there exists an $m \in \M_2(\ZZZ)$ such that $b = ma$. For this to be a partial ordering we need that $a \leq b$ and $b \leq a$ implies $a = b$. This is not clear a priori, but it follows from the concrete description of $\GL_2(\ZZZ)\backslash\M_2(\ZZZ)$ in the following paragraphs.

Note that 
\begin{equation}
\GL_2(\ZZZ)\backslash\M_2(\ZZZ) ~=~ \prod \GL_2(\ZZ_p) \backslash \M_2(\ZZ_p)
\end{equation}
so we can see an element $a \in \GL_2(\ZZZ)\backslash\M_2(\ZZZ)$ as a family
\begin{equation}
(a_p)_p ~\text{ with }a_p \in \GL_2(\ZZ_p)\backslash\M_2(\ZZ_p)\text{ for all primes }p.
\end{equation}
Moreover with this notation we see that $a \leq b$ for $a,b\in\M_2(\ZZZ)$ if and only if $a_p \leq b_p$ for all primes $p$. So we will fix one prime $p$ and look at $\GL_2(\ZZ_p)\backslash\M_2(\ZZ_p)$ in detail.

Because $\ZZ_p$ is a principal ideal domain, every element of $\M_2(\ZZ_p)$ can be brought into its Hermite normal form, by multiplying on the left with elements of $\GL_2(\ZZ_p)$, see \cite[Theorem 22.1]{MacDuffee}. This Hermite normal form is a matrix
\begin{equation}
\begin{pmatrix}
p^k & z \\
0 & p^l
\end{pmatrix}
\end{equation}
with $z = z_0 + z_1 p + z_2 p^2 + \dots$ satisfies $z_i = 0$ for $i \geq l$; but we allow both $k = \infty$ and $l = \infty$ with the convention that $p^\infty = 0$ (in the case $l = \infty$ there is no restriction on $z$). This Hermite normal form is unique whenever the determinant is nonzero (so $k,l$ both finite), see \cite[Theorem 22.2]{MacDuffee}. In this case, we easily find
\begin{equation}
\begin{pmatrix}
p^k & z \\
0 & p^l
\end{pmatrix} \leq \begin{pmatrix}
p^{r} & z' \\
0 & p^{s}
\end{pmatrix}
\end{equation}
if and only if $k \leq r$, $l \leq s$ and $z' \equiv p^{r-k}z \mod p^s$.

Let $a,b \in \GL_2(\ZZ_p)\backslash\M_2(\ZZ_p)$ with nonzero determinant. Then we say that $a$ and $b$ are \emph{adjacent} if
\begin{itemize}
\item $a \leq b$ and $\det(b) = p\det(a)$ (in this case we write $a \to b$);
\item $b \leq a$ and $\det(a) = p\det(b)$ (in this case we write $b \to a$).
\end{itemize}
In this way, we can interpret $\GL_2(\ZZ_p)\backslash\M_2(\ZZ_p)$ as a (directed) graph. Some additional definitions: let $a \in \GL_2(\ZZ_p)\backslash\M_2(\ZZ_p)$, then
\begin{itemize}
\item we define the \emph{level} $\lambda(a)$ as the largest integer such that $p^{\lambda(a)} \mid a$;
\item we define the \emph{niveau} $\nu(a)$ as $\nu(a) = v_p(\det a) - 2\lambda$.
\end{itemize}
If we multiply a matrix by a scalar $n$, then the level increases by $v_p(n)$ and the niveau stays the same. So we could alternatively define the niveau of $a$ as the valuation of the determinant of $\tfrac{1}{N}a$, with $N$ the greatest common divisor for the entries of $a$.

For $a \to b$, we easily compute that either $\lambda(b) = \lambda(a)$ and $\nu(b) = \nu(a) + 1$, or $\lambda(b) = \lambda(a) + 1$ and $\nu(b) = \nu(a) - 1$. Moreover, the latter can only occur for at most one $b$. In Table \ref{table1} we give a complete description of the directed graph structure. Figure \ref{figure1} illustrates the situation for $p=2$. The elements with $\lambda \leq 2$ and $\nu \leq 4$ are drawn with an edge between each two adjacent elements.
\begin{table}[!h]
\bgroup
\renewcommand{\arraystretch}{2}
\scalebox{0.85}{\begin{tabular}{c|c|c|c|c} 
& $\#\left\{\pctext{5em}{$a \to b$ $\lambda(b) = \lambda(a)$} \right\}$
& $\#\left\{\pctext{5em}{$a \to b$ $\lambda(b) > \lambda(a)$} \right\}$
& $\#\left\{\pctext{5em}{$b \to a$ $\lambda(b) = \lambda(a)$} \right\}$
& $\#\left\{\pctext{5em}{$b \to a$ $\lambda(b) < \lambda(a)$} \right\}$
 \\ \hline
\pctext{5em}{$\lambda(a) = 0$ $\nu(a) = 0$} 
& $p+1$ & $0$ & $0$ & $0$ \\ \hline
\pctext{5em}{$\lambda(a) = 0$ $\nu(a) > 0$} 
& $p$ & $1$ & $1$ & $0$ \\ \hline
\pctext{5em}{$\lambda(a) > 0$ $\nu(a) = 0$} 
& $p+1$ & $0$ & $0$ & $p+1$ \\ \hline
\pctext{5em}{$\lambda(a) > 0$ $\nu(a) > 0$} 
& $p$ & $1$ & $1$ & $p$ \\ \hline
\end{tabular}}
\egroup
\caption{Four types in $\GL(\ZZ_p)\backslash\M_2(\ZZ_p)$ and their adjacent elements.}
\label{table1}
\end{table}

\begin{figure}[!h]
\includegraphics[width=0.5\textwidth]{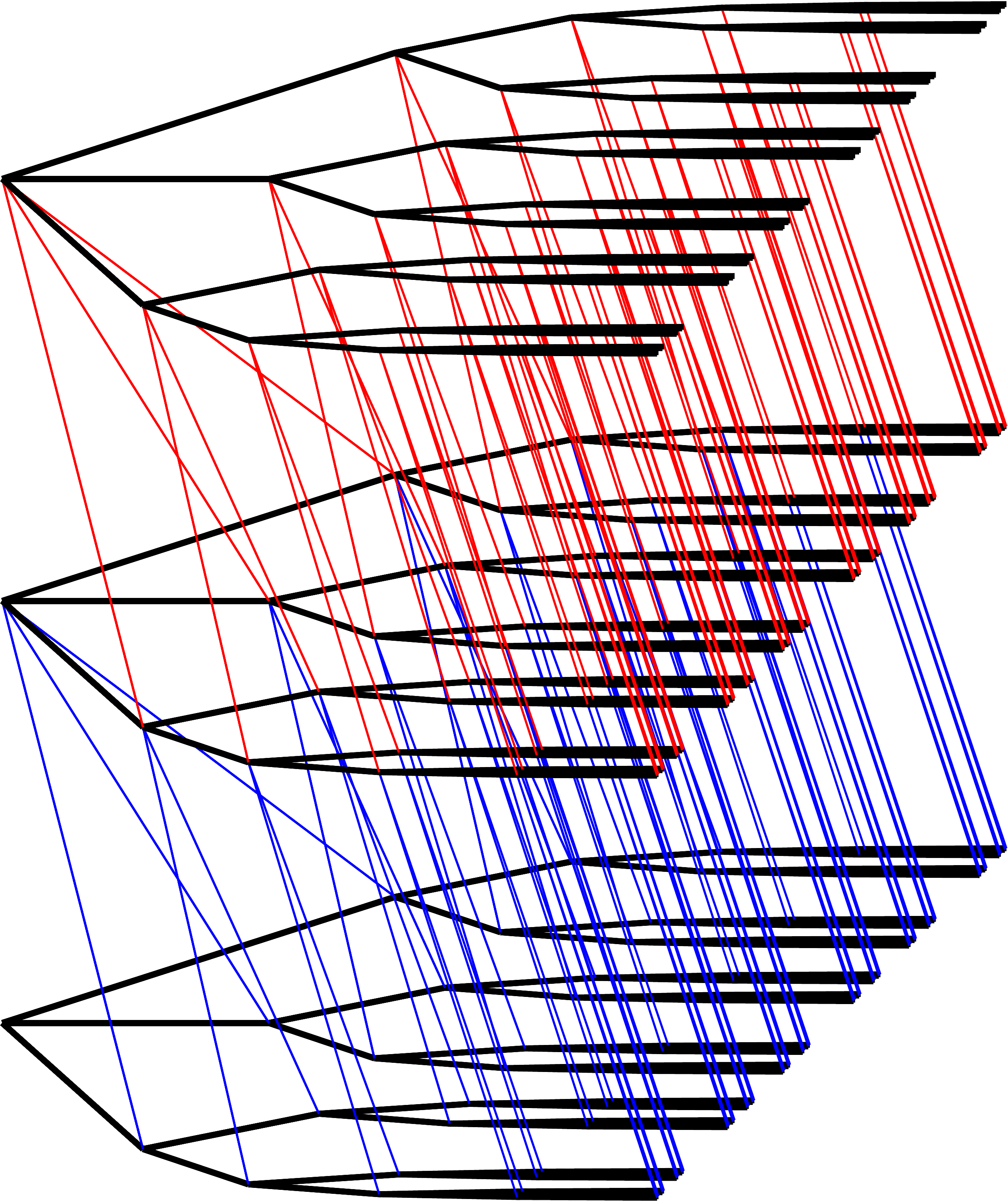}
\caption{A (truncated) picture of $\GL_2(\ZZ_2)\backslash\M_2(\ZZ_2)$.}
\label{figure1}
\end{figure}

For an element of $\GL_2(\ZZ_p)\backslash\M_2(\ZZ_p)$ with zero determinant, the following matrices are unique representatives:
\begin{equation}
\begin{pmatrix}
p^k & z \\
0 & 0
\end{pmatrix} ~\text{ or }~ \begin{pmatrix}
0 & 0 \\
0 & p^l
\end{pmatrix}
\end{equation}
for some $k \in \{0,1,2,3,\dots\}$ and $z\in\ZZ_p$, or $l \in \{0,1,2,\dots\}\cup\{\infty\}$. The poset structure on the latter matrices can be summarized as
\begin{equation*}
\begin{pmatrix}
p^k & z \\
0 & 0
\end{pmatrix} \leq \begin{pmatrix}
p^{k+r} & p^rz \\
0 & 0
\end{pmatrix} \leq \begin{pmatrix}
0 & 0 \\
0 & 0
\end{pmatrix} ~\text{ and }~ \begin{pmatrix}
0 & 0 \\
0 & p^l
\end{pmatrix} \leq \begin{pmatrix}
0 & 0 \\
0 & p^{l+s}
\end{pmatrix} \leq \begin{pmatrix}
0 & 0 \\
0 & 0
\end{pmatrix}.
\end{equation*}
We define the level $\lambda(a)$ again as the largest integer such that $p^{\lambda(a)} \mid a$, or $\lambda(a) = \infty$ for $a$ the zero matrix. For any matrix with zero determinant we write $\nu(a) = \infty$.

It is easy to see from the explicit representatives above that there is a bijection between the nonzero elements with zero determinant and the ``paths'' in $\GL_2(\ZZ_p)\backslash\M_2(\ZZ_p)$, where by ``path'' we mean a subset $\{ a_n \}_{n \in \NN}$ with $\lambda(a_n)=\lambda(a_0)$, $\nu(a_n) = n$ and $a_{n} \leq a_{n+1}$ for all $n \in \NN$. The path associated to $b$ with $\det(b) = 0$ is explicitly given by
\begin{equation}
\{ a \leq b ~:~ \lambda(a) = \lambda(b)  \}.
\end{equation}

This finishes our description of the poset $\GL_2(\ZZ_p) \backslash \M_2(\ZZ_p)$. Note that the determinant map
\begin{equation}
\GL_2(\ZZ_p) \backslash \M_2(\ZZ_p) \longrightarrow \ZZ_p^\times \backslash \ZZ_p
\end{equation}
is also easy to visualize when keeping Figure \ref{figure1} in mind. Moreover, we can identify $\ZZ_p^\times \backslash \ZZ_p$ with $\NN \cup \{\infty\}$ using the map
\begin{equation}
\xi:\NN \cup \{\infty\} \longrightarrow \ZZ_p^\times \backslash \ZZ_p, \qquad n \mapsto p^n
\end{equation}
with the convention $p^\infty = 0$. Note that $\xi(n+m) = \xi(n)\xi(m)$. Taking the product over all primes $p$ gives a description of $\GL_2(\ZZZ)\backslash\M_2(\ZZZ)$ and the corresponding determinant map
\begin{equation*} \label{eq:determinant-map}
\GL_2(\ZZZ)\backslash\M_2(\ZZZ) \longrightarrow \ZZZ^\times\backslash\ZZZ.
\end{equation*}
Here $\ZZZ^\times\backslash\ZZZ$ can be identified with the Steinitz numbers or supernatural numbers
\begin{equation}
\SS = \prod_p \NN\cup\{\infty \}.
\end{equation}
Note that the supernatural numbers $\SS$ turn up as covering for Connes and Consani's Arithmetic Site, see \cite{llb-covers}.

\subsection{The Scott topology} We want to interpret the posets from the previous subsection as topological spaces. The right notion in this setting is the Scott topology.

\begin{definition}[\hspace{-0.1mm}{\cite[Definition 1.3]{Scott}}] \label{def:scott}
Let $X$ be a poset. Then a subset $S \subseteq X$ is open for the \emph{Scott topology} if
\begin{itemize}
\item $S$ is upwards closed (if $x \in S$ and $x \leq y$ then $y \in S$);
\item $S$ is inaccesibly by directed suprema (for any directed supremum $\bigvee_{i \in I} s_i \in S$ we can find an $s_i \in S$).
\end{itemize}
\end{definition}
Note that the Scott topology is originally defined only for complete lattices. We drop this requirement, following e.g.\ \cite{AlHanifi} and \cite{Martin}. One caveat with this more general definition is that it is not functorial: the map from $\NN\cup\{\infty\}$ to $\{0,1\}$ sending natural numbers to zero and $\infty$ to one is order-preserving but not continuous for the Scott topology. This problem does not appear for complete lattices if you require the maps to preserve suprema.

\begin{proposition}
Consider $X = \GL_2(\ZZ)\backslash\M_2^\ns(\ZZ)$ and $\hat{X} = \GL_2(\ZZZ)\backslash\M_2(\ZZZ)$ with the Scott topology. Then:
\begin{enumerate}
\item the open sets for $X$ are precisely the upwards closed sets;
\item the inclusion $X \subset \hat{X}$ is continuous;
\item the inclusion $X \subset \hat{X}$ induces an isomorphism on frames of opens;
\item $X$ is dense in $\hat{X}$;
\item $\hat{X}$ is the sobrification of $X$.
\end{enumerate}
\end{proposition}
\begin{proof} \ 
\begin{enumerate}
\item We have to show that the second part of Definition \ref{def:scott} is automatically satisfied for every upwards closed subset of $X$. So take an upwards closed set $S \subseteq X$ and suppose that $s \in S$ for some $s = \bigvee_{i \in I} s_i$. Then \[\det(s) = \max \{ \det(s_i) : i \in I \}.\]
Take $i \in I$ with $\det(s_i) = \det(s)$, then $s_i = s$ so in particular $s_i \in S$.
\item This follows almost directly from the first statement.
\item We need to prove that the open sets for $\hat{X}$ are the upwards closed sets that are generated by elements of $X$. So take an upwards closed subset $S \subset \hat{X}$ generated by $(g_i)_{i \in I} \in X$. Suppose that $s \in S$ for $s = \bigvee_{j \in J} s_j$. Then there is some $i \in I$ with $g_i \leq s$. Consider the projection
\begin{equation*}
F_p : \GL_2(\ZZZ)\backslash\M_2(\ZZZ) \longrightarrow \GL_2(\ZZ_p)\backslash\M_2(\ZZ_p).
\end{equation*}
Note that the maps $F_p$ preserve suprema \cite[Exercise 3.21.n]{Schechter}. It is enough to find for each prime $p$ some $j \in J$ with $F_p(g_i) \leq F_p(s_j)$, because $F_p(g_i)$ is the identity matrix for almost all primes $p$. We use the abbreviations $\lambda_p$ and $\nu_p$ for $\lambda \circ F_p$ resp.\ $\nu \circ F_p$ with $\lambda$ and $\nu$ the level resp.\ niveau function (see Subsection \ref{ssec:posets}). For each prime $p$ there are three cases.
\begin{itemize}
\item \underline{$\lambda_p(s) = \infty$.} Then we can find an $s_j$ with $\lambda_p(s_j) \geq v_p(\det(g_i))$, but this implies $F_p(g_i) \leq F_p(s_j)$ because $g_i \leq \det(g_i)I_2$ with $I_2$ the identity matrix.
\item \underline{$\lambda_p(s)$ finite, $\nu_p(s) = \infty$.} Without loss of generality, $\lambda_p(g_i) = \lambda_p(s)$. Now take $j \in J$ with $\lambda_p(s_j) = \lambda_p(s)$ and $\nu_p(s_j) \geq \nu_p(g_i)$. Then $F_p(g_i) \leq F_p(s_j)$.
\item \underline{$\lambda_p(s)$ and $\nu_p(s)$ both finite.} Then we can find $j \in J$ with $F_p(s_j) = F_p(s)$, but then also $F_p(g_i) \leq F_p(s_j)$.
\end{itemize}
This shows that $S$ is Scott open. Conversely, let $T \subseteq \hat{X}$ be a Scott open set. Then it is upwards closed, and we need to prove that it is generated by $T \cap X$. Take $s \in T$. Then it is easy to see that $s = \bigvee_{i \in \NN} s_i$ with each $s_i \in X$. Because $T$ is Scott open, we see that $s_i \in T$ for some $i \in \NN$. But now $s_i \in T \cap X$ and $s_i \leq s$.
\item This follows directly from the previous statement.
\item We already know that the inclusion $X \subset \hat{X}$ is continuous and induces an isomorphism on frames of opens. It remains to show that $\hat{X}$ is sober (in other words every irreducible closed subset is the closure of a unique point). It is enough to show this for $\GL_2(\ZZ_p)\backslash\M_2(\ZZ_p)$ with the Scott topology for each prime $p$, because arbitrary products of sober spaces are again sober, see \cite{product-of-sober}. Let $V \subseteq \GL_2(\ZZ_p)\backslash\M_2(\ZZ_p)$ be an irreducible closed subset. In particular it is downwards closed. We claim that it has a unique maximal element $m$, in other words that
\begin{equation*}
V = \{ s \in \GL_2(\ZZ_p)\backslash\M_2(\ZZ_p): s \leq m \}.
\end{equation*}
In the case that $\lambda(s)$ for $s \in V$ is unbounded, it is easy to see that $V$ is the whole space. If $\lambda(s)$ is bounded, write $\lambda_0$ for the maximal value. It is enough to prove that $\lambda(s)=\lambda(s')=\lambda_0$ and $\nu(s) = \nu(s')$ implies $s = s'$ for $s,s' \in V$. But note that $U=\{ v \in V : s \leq v  \}$ and $U'=\{ v \in V:s' \leq v \}$ are two disjoint relatively open sets, which means we can write
\begin{equation*}
V = (V\backslash U) \cup (V\backslash U').
\end{equation*}
This contradicts $V$ being irreducible.
\end{enumerate}
\end{proof}

\begin{corollary}
There is an equivalence of categories
\begin{equation*}
\sh(X) \simeq \sh(\hat{X}).
\end{equation*}
The space of points for this topos is $\hat{X}$.
\end{corollary}

\subsection{Topos-theoretic points for \texorpdfstring{$\setswith{\M_2^\ns(\ZZ)}$}{M2ns(Z)-sets}} \label{ssec:points}

We claim that $\sh(X)$ is a slice topos for $\setswith{\M_2^\ns(\ZZ)}$. We can see $\setswith{\M_2^\ns(\ZZ)}$ as the category of presheaves on the category $\mathtt{M}^\op$ with one object $\ast$ and as morphisms the opposite monoid $\M_2^\ns(\ZZ)^\op$ (composition given by multiplication). Now consider the slice category $\mathtt{M}^\op\slash\ast$. Upto equivalence of categories, we can describe it as follows:
\begin{itemize}
\item The objects are the elements of $X = \GL_2(\ZZ)\backslash\M_2^\ns(\ZZ)$;
\item There is an arrow $n \to m$ whenever $m \leq n$.
\end{itemize}
We can also interpret $\mathtt{M}^\op\slash\ast$ as the category of open sets in $X$ that are generated by one element (and inclusions). Now note that the latter is a basis for the (Scott) topology on $X$. Moreover, there are no nontrivial ways to cover one basis open set with other basis open sets. This brings us to the conclusion
\begin{equation}
\sh(X) \simeq \psh(\mathtt{M}^\op\slash\ast)
\end{equation}
Moreover, the object $\ast$ in $\mathtt{M}^\op$ represents the $\M_2^\ns(\ZZ)$-set $M = \M_2^\ns(\ZZ)$ with left action given by multiplication. So we have an equivalence
\begin{equation}
\psh(\mathtt{M}^\op\slash\ast) ~\simeq~ \setswith{\M_2^\ns(\ZZ)}\slash M.
\end{equation}
Because we see $\sh(X)$ as a slice topos for $\setswith{\M_2^\ns(\ZZ)}$, there is an induced geometric morphism
\begin{equation}
\begin{tikzcd}
\sh(X) \ar[r,"{\pi_*}",bend right] & \setswith{\M_2^\ns(\ZZ)} \ar[l,bend right,"{\pi^*}"'].
\end{tikzcd}
\end{equation}
This geometric morphism induces a continuous map from the space of points $\hat{X}$ of $\sh(X)$ to the space of points of $\setswith{\M_2^\ns(\ZZ)}$. Moreover, by \cite[Observation 3.10]{nLab-slice-topos}, there is a point of $\sh(X)$ for every couple $(p,x)$ with $p$ a point of $\setswith{\M_2^\ns(\ZZ)}$ and $x \in p^*M$, and this point is sent to $p$ under the above map. In particular the map 
\begin{equation}
\pi : \hat{X} \to \mathsf{Pts}(\setswith{\M_2^\ns(\ZZ)})
\end{equation}
is surjective, if we prove that $p^*M$ is nonempty for every point $p$ of $\setswith{\M_2^\ns(\ZZ)}$. So suppose that $p^*M = \varnothing$. Then by the adjunction property $p_*S$ is the one-element set $1$ for each set $S$. From this it follows that $p^*1 = \varnothing$, but this contradicts $p^*$ preserving finite limits. As a conclusion:
\begin{proposition}
There is a surjective continuous map
\[
\pi: \hat{X} \relbar\joinrel\twoheadrightarrow \mathsf{Pts}(\setswith{\M_2^\ns(\ZZ)}).
\]
\end{proposition}
Now we need to determine when $\pi(a) = \pi(b)$ for $a,b \in \hat{X}$. We start by giving a description of the map $\pi$. Let $x \in \hat{X}$. Then for a sheaf $\F$ on $X$, we get $x^*\F = \varprojlim_{x \in U} \F(U)$. Clearly, it is enough to take the projective limit over only the open sets $U$ generated by one element. Note that $\pi(x)^* = x^*\circ \pi^*$. Moreover, $\pi^*$ is given by
\begin{equation}
(\pi^*N)(U) = N
\end{equation}
for $N$ an arbitrary $\M_2^\ns(\ZZ)$-set. Further, if $U$ is generated by $m$ and $V$ by $n = am$, then the restriction $(\pi^*N)(U) \to (\pi^*N)(V)$ is given by left multiplication with $a$.

This leads to the following description of $\pi(x)$: for each $\M_2^\ns(\ZZ)$-set $N$,
\begin{equation}
\pi(x)^*N = \varinjlim_{m \leq x} N,
\end{equation}
where the transition map along $m \leq m'$, with $m' = am$, is given by multiplication by $a$. Note that the transition maps are not necessarily equivariant (we take the projective limit as sets).

\begin{theorem} \label{thm:points}
The points of $\setswith{\M_2^\ns(\ZZ)}$ are given by
\[
\GL_2(\ZZZ)\backslash\M_2(\AA_f)\slash\GL_2(\QQ)
\]
where $\AA_f = \ZZZ \otimes \QQ$ are the finite adeles. Moreover, there are in bijective correspondence with the isomorphism classes of subgroups $\ZZ^2 \subseteq A \subseteq \QQ^2$.
\end{theorem}
\begin{proof}
Consider the map $\pi:\hat{X} \twoheadrightarrow \mathsf{Pts}(\setswith{\M_2^\ns(\ZZ)})$ as above. To $x \in X$ we associate the subgroup $\ZZ^2 \subseteq A_x \subseteq \QQ^2$ given by
\begin{equation}
A_x = \pi(x)^*\ZZ^2 = \varinjlim_{m \leq x} \ZZ^2
\end{equation}
where $\ZZ^2$ is seen as column vectors equipped with the standard left action of $\M_2^\ns(\ZZ)$. The inclusion in $\QQ^2$ is the one induced by taking the standard inclusion $\ZZ^2 \subset \QQ^2$ for $m$ the identity matrix.

Now suppose that $\pi(x) = \pi(y)$, in other words there is a natural isomorphism of functors $\pi(x)^* \simeq \pi(y)^*$. Then in particular there is an isomorphism
\begin{equation}
A_x \stackrel{\sim}{\to} A_y
\end{equation}
and by naturality this isomorphism preserves addition and negation (so it is a group isomorphism). Using extension of scalars to $\QQ$, we can find some $g \in \GL_2(\QQ)$ such that $A_y = g\cdot A_x$, or rather
\begin{equation}
h^{-1} \cdot A_y = N^{-1} \cdot A_x
\end{equation}
for $h = Ng \in \M_2^\ns(\ZZ)$. This in turn implies $yh = xN$ or $yg = x$, in other words $x$ and $y$ determine the same class in 
\begin{equation}
\GL_2(\ZZZ)\backslash\M_2(\AA_f)\slash\GL_2(\QQ).
\end{equation}
Conversely, take $x,y \in \hat{X}$ with $x = yg$. Then it is easy to construct a natural isomorphism $\pi(x)^* \simeq \pi(y)^*$.

We still need to show that for any subgroup $\ZZ^2 \subseteq A \subseteq \QQ^2$ we can find an $x \in \hat{X}$ such that $A = A_x$. We use the fact that $A$ can be written as an inductive limit of finitely generated subgroups of rank 2 containing $\ZZ^2$. These finitely generated subgroups can each be written as $m^{-1}\ZZ^2$ for some $m \in \M_2^\ns(\ZZ)$. But then $A = A_x$ for $x$ the supremum of all these $m$'s in the poset $\hat{X}$.
\end{proof}

\begin{remark}
The idea of using finite adeles to describe subgroups $\ZZ \subseteq A \subseteq \QQ^2$ up to isomorphism is not new. The standard approach is the one introduced by Mal'cev \cite{Malcev} in 1938,  see \cite[Theorem 93.4]{Fuchs}. However, the description from \mbox{Theorem \ref{thm:points}} is a bit different and was not found by the author of this paper in previous works. 

Note that both the original description by Mal'cev and the variation above are rather unpractical. For two matrices $a,b \in \M_2(\ZZZ)$ it is in general very difficult to determine if they represent the same element of $\GL_2(\ZZZ)\backslash\M_2(\AA_f)\slash\GL_2(\QQ)$.
\end{remark}

\subsection{Automorphisms of \texorpdfstring{$\setswith{\M_2^\ns(\ZZ)}$}{M2ns(Z)-sets}} \label{ssec:automorphisms}

The goal of this subsection is to show that the group of automorphisms of $\setswith{\M_2^\ns(\ZZ)}$ (up to natural isomorphism) is rather small: it is isomorphic to the group of characters $\hom(\QQ^\times,\{1,-1\})$. Intuitively, it is not surprising that it is a small group, when taking into account the result of Stephenson \cite{Stephenson} that all monoid automorphisms of $\M_2(\ZZ)$ are inner. But Stephenson's proof uses the idempotents in $\M_2(\ZZ)$; if we consider the submonoid $\M_2^\ns(\ZZ)$, then there are more automorphisms, as shown in the following proposition.

\begin{proposition} \label{prop:auto}
Let $\vartheta: \M_2^\ns(\ZZ) \to \M_2^\ns(\ZZ)$ be an automorphism of monoids. Then we can find $g \in \GL_2(\ZZ)$ and a character $\chi:\QQ^\times\to\{1,-1\}$ such that
\[
\vartheta(a) = \chi(\det(a))\, gag^{-1}.
\]
In particular, $\vartheta$ preserves the determinant.
\end{proposition}
\begin{proof}
The groupification of $\M_2^\ns(\ZZ)$ is $\GL_2(\QQ)$, so $\vartheta$ is the restriction of a group automorphism $\hat{\vartheta} : \GL_2(\QQ) \to \GL_2(\QQ)$. This $\hat{\vartheta}$ can be written as
\begin{equation}
\hat{\vartheta}(a) = \hat{\chi}(a)\, gag^{-1}
\end{equation}
with $\chi:\GL_2(\QQ) \to \QQ^\times$ a character and $g \in \GL_2(\QQ)$ (use \cite[Theorem 1]{Hua} and keep in mind that $a \mapsto \det(a) \, (a^t)^{-1}$ is conjugation by
\[
\begin{pmatrix}
0 & 1 \\
-1 & 0
\end{pmatrix}
\]
so this is an inner automorphism).  Note that $\hat{\chi}$ is necessarily trivial on the commutator subgroup $\SL_2(\QQ)$ of $\GL_2(\QQ)$, so we can write $\hat{\chi}(a) = \chi(\det(a))$ for some character $\chi : \QQ^\times \to \QQ^\times$.

Further, $\vartheta$ induces a poset automorphism
\begin{equation}
\bar{\vartheta} ~:~ \GL_2(\ZZ)\backslash\M_2^\ns(\ZZ) \longrightarrow \GL_2(\ZZ)\backslash\M_2^\ns(\ZZ)
\end{equation}
In particular, it preserves the elements that have exactly one strictly smaller element. These correspond to the elements $a \in \M_2^\ns(\ZZ)$ with $\det(a) = \pm p$ for some prime number $p$. So if $\det(a) = \pm p$ then up to sign $\det(\vartheta(a)) = q$ with $q$ prime. Moreover
\begin{equation}
\det(\vartheta(a)) = \chi(\det(a))^2\, \det(a)
\end{equation}
so in fact $\det(\vartheta(a)) = \pm p$ (with the same sign) and $\chi(\det(a))^2 = 1$. Using the Smith normal form, we see that $\M_2^\ns(\ZZ)$ is generated by units and matrices of prime determinant. As a corollary $\det(\vartheta(a)) = \det(a)$ and $\chi(\det(a))^2 = 1$ for any $a \in \M_2^\ns(\ZZ)$. We still need to deduce that $g \in \GL_2(\ZZ)$ from $\hat{\vartheta}$ sending $\M_2^\ns(\ZZ)$ to itself. Note that $gag^{-1} \in \M_2^\ns(\ZZ)$ for all $a \in \M_2^\ns(\ZZ)$. Further, for any $b \in \M_2(\ZZ)$ we can find $\lambda \in \ZZ$ big enough such that $b + \lambda I_2$ is in $\M_2^\ns(\ZZ)$. This implies that $gbg^{-1} \in \M_2(\ZZ)$ and in this way conjugation by $g$ defines a ring automorphism of $\M_2(\ZZ)$. This shows $g \in \GL_2(\ZZ)$.
\end{proof}

In a cancellative monoid $M$ being left invertible is equivalent to being right invertible. The group of (left and right) invertible elements will be denoted by $M^\times$. As in the case of groups there is a (group) morphism
\begin{gather*}
M^\times \longrightarrow \aut(M) \\
g ~\mapsto~ \iota_g
\end{gather*}
with $\iota_g(m) = gmg^{-1}$. The image is the normal subgroup of inner automorphisms, denoted by $\inn(M) \subseteq \aut(M)$. We write
\[
\out(N) = \frac{\aut(M)}{\inn(M)}
\]
for the group of outer automorphisms. With these notations we can formulate an immediate corollary to Proposition \ref{prop:auto}.

\begin{corollary}
There is a group isomorphism
\[
\out(\M_2^\ns(\ZZ)) ~\cong~ \Hom(\QQ^\times,\{1,-1\}).
\]
\end{corollary}

Note that every $f \in \Hom(\QQ^\times,\{1,-1\})$ is uniquely determined by choosing a value for $f(-1)$ and for $f(p)$ for all primes $p$.

We now want to determine the topos automorphisms of $\setswith{\M_2^\ns(\ZZ)}$. We first present a criterion describing which topos automorphisms of $\setswith{M}$ (with $M$ an arbitrary monoid) arise from monoid automorphisms of $M$.

\begin{proposition}
Let $M$ be an arbitrary monoid and let $\Theta : \setswith{M} \to \setswith{M}$ be an equivalence such that $G\Theta \simeq G$ for $G : \setswith{M} \to \sets$ the forgetful functor. Then $\Theta \simeq \vartheta^*$ for some monoid automorphism $\vartheta : M \to M$.
\end{proposition}
\begin{proof} For $\theta$ an automorphism of $M$, the equivalence $\theta^*$ satisfies $G\theta^* = G$ and can be reconstructed from the monoid map
\begin{gather*}
\nat(G,G) ~\to~ \nat(G,G) \\
\alpha_m ~\mapsto~ \alpha_m \vartheta^* = \alpha_{\vartheta(m)}.
\end{gather*}
Here $\alpha_m$, $m \in M$ is the natural transformation given by $\alpha_m(n) = m \cdot n$ for all $M$-sets $N$ and $n \in N$ (every natural transformation $G \Rightarrow G$ is of this form). We denote horizontal composition for natural transformations by juxtaposition and we denote vertical composition with $\circ$. Note that any monoid map $\nat(G,G) \to \nat(G,G)$ is of the form $\alpha_m \mapsto \alpha_{\vartheta(m)}$ for some $\vartheta$.

Now take some $\Theta$ as in the proposition an take an equivalence $\varphi : G\Theta \Rightarrow G$. Then we can construct a monoid map
\begin{align*}
\nat(G,G) ~\to~ \nat(G,G) \\
\alpha_m \mapsto \varphi \circ \alpha_m\Theta \circ \varphi^{-1}.
\end{align*}
We can find a monoid automorphism $\vartheta : M \to M$ such that
\begin{equation}
\alpha_{\theta(m)} = \varphi \circ \alpha_m\Theta \circ \varphi{-1}
\end{equation}
in other words, such that the diagram
\begin{equation}
\begin{tikzcd}
\Theta N \ar[r,"{\varphi}"] \ar[d,"{\alpha_m}"'] & N \ar[d,"{\alpha_{\vartheta(m)}}"] \\
\Theta N' \ar[r,"{\varphi}"'] & N'
\end{tikzcd}
\end{equation}
commutes for any $M$-sets $N$ and $N'$ and any $m \in M$. But this means that $\varphi$ is equivariant, so it defines a natural isomorphism $\varphi : \Theta \Rightarrow \vartheta^*$.
\end{proof}

\begin{corollary} \label{cor:monoid-topos-auto}
Let $M$ be a cancellative monoid. Then the topos automorphisms $\Theta$ of $\setswith{M}$ up to natural isomorphism satisfying $G\Theta \simeq G$ form a group isomorphic to $\out(M)$.
\end{corollary}
\begin{proof}
We already showed that every autoequivalence $\Theta$ with $G\Theta \simeq G$ is induced by an automorphism of $M$ (up to natural isomorphism). We still need to determine when $\vartheta^* \simeq \zeta^*$ for $\vartheta$, $\zeta$ two automorphisms of $M$. If $\zeta = \iota_g \circ \vartheta$ for some $g \in M^\times$, then left multiplication by $g$ defines a natural isomorphism $\vartheta \Rightarrow \zeta$. Conversely, if $\varphi:\vartheta \Rightarrow \zeta$ is a natural isomorphism, then by functoriality $\varphi(ma) = \varphi(m)a$ for all $m,a \in M$. So $\varphi(m) = gm$ for some $g \in M^\times$. Because $\varphi$ is equivariant, we see that $\zeta = \iota_g \circ \vartheta$. In conclusion, $\vartheta$ and $\zeta$ are naturally isomorphic if and only if they differ by an inner automorphism.
\end{proof}

We now return to the case $M = \M_2^\ns(\ZZ)$. We first show that the criterion from Corollary \ref{cor:monoid-topos-auto} is satisfied for all topos automorphisms.

\begin{lemma} \label{lmm:auto-preserves-forget}
Let $$\Theta : \setswith{\M_2^\ns(\ZZ)} \longrightarrow \setswith{\M_2^\ns(\ZZ)}$$ be an autoequivalence. Then $G\Theta \simeq G$.
\end{lemma}
\begin{proof}
Note that $G \simeq p^*$ with $p$ the point corresponding to the identity matrix. Similarly, $G\Theta$ has a right adjoint and preserves finite limits so $G\Theta \simeq q^*$ for some point $q$. For $N$ an $\M_2^\ns(\ZZ)$-set, we can write
\begin{equation}
G\Theta N \cong q^*N = \varinjlim_{m \leq x} N
\end{equation}
for some $x \in \GL_2(\ZZZ)\backslash\M_2(\ZZZ)$. We claim that $q^* \simeq p^* \simeq G$. To show this, we take $N = \{0,1\}$ with the $\M_2^\ns(\ZZ)$-action
\begin{equation}
m \cdot x = \begin{cases}
x &\text{if }\det(m) = 1 \\
0 &\text{if }\det(m) \neq 1
\end{cases}.
\end{equation}
We compute
\begin{equation}
\varinjlim_{m \leq x} N = \begin{cases}
N &\text{if }x\in X \\
1 &\text{if }x\in \hat{X}\setminus X
\end{cases}
\end{equation}
(recall the notations $X = \GL_2(\ZZ)\backslash\M_2^\ns(\ZZ)$ and $\hat{X} = \GL_2(\ZZZ)\backslash\M_2(\ZZZ)$). Because $\Theta$ is an equivalence, it cannot send $N$ to $1$ (if so, $\Theta^{-1}$ would not preserve the terminal object). So $x \in X$ and because $X$ is precisely the orbit of the indentity matrix under the $\GL_2(\QQ)$-action, we see that $G\Theta \simeq q^* \simeq p^* \simeq G$.
\end{proof}

We now combine Corollary \ref{cor:monoid-topos-auto} with Lemma \ref{lmm:auto-preserves-forget}. In order to determine the topos automorphisms of $\setswith{\M_2^\ns(\ZZ)}$ we just need to compute $\out(\M_2^\ns(\ZZ))$. By Proposition \ref{prop:auto}, each automorphism $\vartheta$ can be written as $\vartheta(a) = \chi(a)\, gag^{-1}$ and $\chi_1(a)\,gag^{-1} = \chi_2(a)\,hah^{-1}$ implies that $g = h$ and $\chi_1 = \chi_2$. So we can identify the outer automorphisms with the characters $\M_2^\ns(\ZZ) \to \{1,-1\}$. This leads us to the following description.

\begin{theorem}
The group of topos automorphisms of $\setswith{\M_2^\ns(\ZZ)}$ up to natural isomorphism can be identified with the group of characters
\[
\chi : \QQ^\times \to \{1,-1\}
\]
under pointwise multiplication. 
\end{theorem}

\subsection{The \texorpdfstring{$ax+b$}{ax+b} monoid} \label{ssec:pz}

In a recent paper of Connes and Consani \cite{ConnesConsaniLifts}, the subsets
\begin{equation}
\bar{P}(R)=\left\{ \begin{pmatrix}
a & b \\
0 & 1 
\end{pmatrix} ~:~ a \in R,~ b \in R  \right\}
\end{equation}
are introduced, for an arbitrary commutative ring $R$. They are used to study \emph{parabolic} $\QQ$-lattices \cite[Definition 6.1, p.~50]{ConnesConsaniLifts}. For example, the parabolic $\QQ$-lattices up to commensurability are given by
\begin{equation}
\mathcal{C}^0_\QQ = P^+(\QQ)\backslash \left(\bar{\P}(\AA_f) \times \P^+(\mathbb{R}) \right),
\end{equation}
where the superscript $+$ means that we take the subset of matrices $\begin{pmatrix}
a & b \\
0 & 1
\end{pmatrix}$ with $a > 0$ \cite[Theorem 6.1.(ii), p.~52]{ConnesConsaniLifts}.

In this subsection we will study the submonoid $\bar{\P}^\ns(\ZZ) \subset \bar{\P}(\ZZ)$ consisting of matrices with nonzero determinant, and the associated topos $\setswith{\bar{\P}^\ns(\ZZ)}$. We will show that the topos has three nice properties as a setting for the Riemann Hypothesis 
\begin{enumerate}
\item its topos points are given by $\ZZZ^\times\backslash\AA_f/\QQ^\times$, so by \cite{llb-covers} the topos points are the same as for the underlying topos of the Arithmetic Site of \cite{ConnesConsani};
\item its zeta function is the Riemann zeta function $\zeta(s)$;
\item its group of topos automorphisms is $\ZZ/2\ZZ$.
\end{enumerate}

So $\setswith{\bar{\P}^\ns(\ZZ)}$ resembles the Arithmetic Site (without structure sheaf), but this time there are only two automorphisms.

The units of the monoid $\bar{\P}^\ns(\ZZ)$ are given by the matrices $\begin{pmatrix}
\pm 1 & b \\
0 & 1
\end{pmatrix}$, and it is easy to see that the elements of the quotient $Y = \bar{\P}^\ns(\ZZ)^\times \backslash \bar{\P}^\ns(\ZZ)$ have unique representatives of the form $\begin{pmatrix}
a & 0 \\
0 & 1
\end{pmatrix}$, $a > 0$. There is an obvious inclusion $Y \subset X$, with $X = \GL_2(\ZZ)\backslash\M_2^\ns(\ZZ)$ as before. This inclusion puts a partial order on $Y$, and it is continuous for the induced topologies (with upwards closed sets as opens). Let $p^* : \setswith{\bar{\P}^\ns(\ZZ)} \longrightarrow \sets$ be a topos-theoretic point. Then completely analogously to the situation in Subsection \ref{ssec:points}, we can show that $p^* = x^*\pi^*$, where $x^*$ is a point of $\sh(Y)$ and $\pi^*$ is the pullback part of the geometric morphism
\begin{equation}
\begin{tikzcd}[column sep=large]
\sh(Y) \ar[r,bend left,"{\pi_!}"] \ar[r,bend right,"{\pi_*}"] 
& \setswith{\bar{\P}^\ns(\ZZ)} \ar[l,"{\pi^*}"']
\end{tikzcd}.
\end{equation}
So the topos points of $\setswith{\bar{\P}^\ns(\ZZ)}$ are exactly the topos points of $\sh(Y)$
modulo some equivalence relation. But the commutative diagram
\begin{equation}
\begin{tikzcd}
\sh(Y) & \sh(X) \ar[l] \\
\setswith{\bar{\P}^\ns(\ZZ)} \ar[u] & \setswith{\bar{\P}^\ns(\ZZ)} \ar[u] \ar[l]
\end{tikzcd}
\end{equation}
induces a commutative diagram between the spaces of points
\begin{equation}
\begin{tikzcd}
\hat{Y} \ar[r] \ar[d] & \hat{X} \ar[d] \\
\mathsf{Pts}(\setswith{\bar{\P}^\ns(\ZZ)}) \ar[r] & \mathsf{Pts}(\setswith{\M_2^\ns(\ZZ)})
\end{tikzcd},
\end{equation}
where $\hat{Y}$ and $\hat{X}$ are the sobrifications of $Y$ resp.\ $X$. So if two points $y_1, y_2 \in \hat{Y}$ are equivalent, then there is some $g \in \GL_2(\QQ)$ such that $y_2 = y_1 g$. Conversely, it is easy to see that if $y_2 = y_1g$, then $y_1$ and $y_2$ define the same point of $\setswith{\bar{\P}^\ns(\ZZ)}$.

Recall that the action of $\GL_2(\QQ)$ on $\hat{X}$ was defined using the identification
\begin{equation}
\hat{X} = \GL_2(\ZZZ)\backslash\M_2(\ZZZ).
\end{equation}
Under this identification we find
\begin{equation}
\hat{Y} = \left\{ \begin{pmatrix}
z & 0 \\
0 & 1
\end{pmatrix} ~:~  z \in \ZZZ^\times\backslash\ZZZ  \right\} ~\subset~ \hat{X}.
\end{equation}
To determine the points of $\setswith{\bar{\P}^\ns(\ZZ)}$, we have to quotient out the action of $\GL_2(\QQ)$. We then get the following theorem.

\begin{theorem}
The topos points of $\setswith{\bar{\P}^\ns(\ZZ)}$ are given by
\[
\ZZZ^\times \backslash \ZZZ / \NN_+^\times ~=~ \ZZZ^\times \backslash \AA_f / \QQ^\times.
\]
In particular, by \cite{llb-covers} they agree with the topos points of $\setswith{\NN_+^\times}$, the underlying topos for the Arithmetic Site of \cite{ConnesConsani}.
\end{theorem}

The embedding $Y \subset X$ allows us to associate a zeta function to $\setswith{\bar{\P}^\ns(\ZZ)}$, namely
\begin{equation}
\zeta_{\bar{\P}}(s) = \sum_{a \in Y} \det(a)^{-s}.
\end{equation}
Clearly $\zeta_{\bar{\P}}(s)$ is the Riemann zeta function $\zeta(s)$.

Now we show that $\setswith{\bar{\P}^\ns(\ZZ)}$ has only two topos automorphisms. We first determine the outer automorphisms of $\bar{\P}^\ns(\ZZ)$.

\begin{proposition}
The outer automorphism group of $\bar{\P}^\ns(\ZZ)$ is isomorphic to $\ZZ/2\ZZ$.
\end{proposition}
\begin{proof}
The nonzero integers will be considered as a submonoid, where $n \in \ZZ$, $n \neq 0$, corresponds to the matrix $\begin{pmatrix}
n & 0 \\ 
0 & 1
\end{pmatrix}$. The nonzero integers generate $\bar{\P}^\ns(\ZZ)$, together with the matrix $t \in \begin{pmatrix}
1 & 1 \\
0 & 1
\end{pmatrix}$ and its inverse. Every element has a unique representation in the form $t^b a$, with $a,b \in \ZZ$, $a \neq 0$. Moreover, $nt = t^n n$. In is now easy to see that, if $\alpha$ is an automorphism of $\bar{\P}^\ns(\ZZ)$, then $\alpha(n) = t^{k(n)}n$, for some $k(n) \in \ZZ$ depending on $n$. Because $\alpha(n)$ and $\alpha(2)$ commute, we find
\begin{equation}
k(n) = (n-1)k(2).
\end{equation}
This implies $\alpha(n) = t^{-k(2)}nt^{k(2)}$, so up to an inner automorphism we can assume that $\alpha(n) = n$ for all $n \in \ZZ$, $n \neq 0$. It is now clear that either
\begin{equation}
\alpha(t^k n) = t^k n
\end{equation}
or
\begin{equation}
\alpha(t^k n) = t^{-k}n.
\end{equation}
\end{proof}

\begin{corollary}
The group of topos automorphisms of $\setswith{\bar{\P}^\ns(\ZZ)}$ (up to natural isomorphism) is isomorphic to $\ZZ/2\ZZ$. 
\end{corollary}
\begin{proof}
By Corollary \ref{cor:monoid-topos-auto}, we need to show that any topos automorphism $\Theta$ of $\setswith{\bar{\P}^\ns(\ZZ)}$ satisfies $G \Theta \simeq G$, where $G$ is the forgetful functor. The proof of this fact is completely analogous to the proof of Lemma \ref{lmm:auto-preserves-forget}.
\end{proof}

\section{Relation to Conway's Big Picture} \label{sec:conway}

Conway's big picture (introduced in \cite{Conway}) is the graph with vertex set $\QQ_+ \times \QQ/\ZZ$ and edges defined by a hyper-distance $\delta$.

We use the notations from \cite{llbMM}: the big picture is denoted by $\mathfrak{P}$, and for each $X = (M,\tfrac{g}{h}) \in \QQ_+ \times \QQ/\ZZ$ we consider the matrices
\begin{equation} \label{eq:alphaX}
\alpha_X = \begin{pmatrix}
M & \tfrac{g}{h} \\
0 & 1
\end{pmatrix} \in \Gamma\backslash\GL_2^+(\QQ)
\end{equation}
where $\Gamma = \PSL_2(\ZZ)$ is the modular group and $\GL_2^+(\QQ)$ the subgroup of $\GL_2(\QQ)$ consisting of the matrices with positive determinant. The hyper-distance $\delta$ is then given by
\begin{equation}
\delta(X,Y) = \det(\alpha_{XY} \alpha_X  \alpha_Y^{-1})
\end{equation}
with $\alpha_{XY}$ the smallest rational number such that $\alpha_{XY}\alpha_X\alpha_Y^{-1} \in \M_2^+(\ZZ)$; further, there is an edge between $X$ and $Y$ whenever $\delta(X,Y)$ is a prime number (see \cite[\mbox{p. 7}]{llbMM} for all this).

We claim that we can embed $\mathfrak{P}$ as a full subgraph of $\GL_2(\ZZ)\backslash\M_2^\ns(\ZZ)$. The embedding is given on vertices by
\begin{equation} \label{eq:betaX}
\QQ_+ \times \QQ/\ZZ,~~(M,\tfrac{g}{h})=X \mapsto \begin{pmatrix}
MN & \tfrac{g}{h}N \\
0 & N
\end{pmatrix} = \beta_X
\end{equation}
where $N \in \NN_+$ is minimal such that $\tfrac{g}{h}N \in \ZZ$. We can assume $0\leq \tfrac{g}{h} < 1$ and then $\beta_X$ is in Hermite normal form so the mapping is injective. Further, the greatest common divisor of the entries of $\beta_X$ is $1$ (i.e.\ the entries are coprime) and conversely every $a \in \GL_2(\ZZ)\backslash\M_2^\ns(\ZZ)$ with this property can be written as $a = \beta_X$ for some $X \in \QQ_+ \times \QQ/\ZZ$. In the notations of Section \ref{sec:topos},
\begin{equation}
\mathfrak{P} ~=~ \{ a \in \GL_2(\ZZ)\backslash\M_2^\ns(\ZZ) ~:~ \lambda_p(a) = 0 \text{ for all primes }p  \}
\end{equation}
where $\lambda_p(a)$ is the level of $a$ at prime $p$.

In the following proposition, we will see that the poset structure on $\mathfrak{P}$ as a subset of $\GL_2(\ZZ)\backslash\M_2^\ns(\ZZ)$ is the same as the poset structure on the big picture as introduced in \cite[Definition 1]{llbMM}, i.e.\ the one given by
\begin{equation} \label{eq:poset-bp}
X \leq Y ~\text{ iff }~ \delta(1,Y)=\delta(X,Y)\delta(1,X).
\end{equation}
A fortiori, the edges of the underlying graphs are the same.

\begin{proposition} \label{prop:bigpicture}
Consider the hyper-distance $\tilde{\delta}$ on $\GL_2(\ZZ)\backslash\M_2^\ns(\ZZ)$ given by
\[
\tilde{\delta}(a,b) = \det(a')\det(b')
\]
where $x = a \wedge b$ and $a = a'x$, $b = b'x$. Then:
\begin{enumerate}
\item $\log \tilde{\delta}(a,b)$ is the weighted distance between $a$ and $b$ where an edge $x \to y$ with $\det(y) = p\det(x)$ has weight $\log(p)$;
\item $\tilde{\delta}(x,y) = \delta(x,y)$ for $x,y \in \mathfrak{P}$;
\item the poset structure on $\mathfrak{P}$ as a subset of $\GL_2(\ZZ)\backslash\M_2^\ns(\ZZ)$ agrees with the poset structure from (\ref{eq:poset-bp}) above \cite[Definition 1]{llbMM};
\item $\mathfrak{P}$ is a fundamental domain for the monoid action of $\NN_+^\times$ on $\GL_2(\ZZ)\backslash\M_2^\ns(\ZZ)$ by scalar multiplication.
\end{enumerate}
\end{proposition}
\begin{proof} \ 
\begin{enumerate}
\item We denote the weighted distance between $a$ and $b$ by $d(a,b)$. Note that
\[
d(a,b) = \sum_p d_p(a_p,b_p)
\]
where $a_p$ and $b_p$ are the projections on $\GL_2(\ZZ_p)\backslash\M_2^\ns(\ZZ_p)$ of $a$ resp.\ $b$, and $d_p$ is the weighted distance function in $\GL_2(\ZZ_p)\backslash\M_2^\ns(\ZZ_p)$, where every edge has weight $\log(p)$. As a weighted graph, $\GL_2(\ZZ_p)\backslash\M_2^\ns(\ZZ_p)$ can be identified with
\[
X_p = \{a \in \GL_2(\ZZ)\backslash\M_2^\ns(\ZZ) ~:~ a_q = 1 \text{ for all }q \neq p\}.
\]
It follows from the above that it is enough to prove the statement for $a,b \in X_p$. First assume that $x = a \wedge b = 1$, $a \neq 1$ and $b \neq 1$. We claim that then $a,b \in \mathfrak{P}$, i.e.\ $p \nmid a$ and $p \nmid b$. Indeed, suppose $p \mid a$. Any divisor $y \leq b$ with $\det(y) = p$ then also divides $a$. This shows $b = 1$, a contradiction. So $p \nmid a$ and analogously $p \nmid b$. Take a path of minimal length from $a$ to $b$. We can assume that this path does not leave $\mathfrak{P}$, so it is of length $\det(a)\det(b)$. Now suppose that $x = a\wedge b \neq 1$. Again we take a path of minimal length from $a$ to $b$ and we can assume that this path does not leave
\[
\uparrow x = \{ a \in X_p ~:~ a \geq x \}.
\]
Multiplication by $x^{-1}$ on the right is an isometry from $\uparrow x$ to $X_p$, and replaces $a$ by $a'$, $b$ by $b'$ and $x$ by $1$. From the previous case we find
\[
d(a,b) = d(a',b') = \det(a')\det(b').
\]
\item This follows directly from (1).
\item It is enough to show that
\[
a \leq b ~\text{ iff }~ \tilde{\delta}(1,b) = \tilde{\delta}(a,b)\tilde{\delta}(1,a).
\]
This easily follows from $(1)$, for example by induction on the number of prime divisors of $\tilde{\delta}(x,y)$ (counted with multiplicity).
\item This is clear from the description of $\mathfrak{P}$ as consisting of the matrices for which the entries are coprime.
\end{enumerate}
\end{proof}

From now on, we use the notation $\fM = \GL_2(\ZZ) \backslash \M_2^\ns(\ZZ)$. We can associate to it the zeta function
\begin{equation}
\zeta_\fM(s) = \sum_{x \in \fM} \det(x)^{-s}.
\end{equation}
It is proved in \cite{Saito} that
\begin{equation}
\zeta_\fM(s) = \zeta(s) \zeta(s-1).
\end{equation}
Note that this is the same as the Hasse-Weil zeta function for $\PP^1_\ZZ$.

\begin{remark}
Saito in \cite{Saito} considers $\GL_n(R)\backslash\M_n^\ns(R)$ for $n$ a natural number and $R$ a principal ideal domain, and shows that the zeta function is equal to
\begin{equation}
\zeta_R(s)\zeta_R(s-1)\cdots\zeta_R(s-n+1).
\end{equation}
(Saito uses the notation $P_{\M(n,R)^\times,\deg}(\exp(-s))$ for this zeta function.) In our case, it follows directly from the Hermite normal form that the number of elements in $\fM$ with determinant $n$ is given by $\sigma(n)$, so
\begin{equation}
\zeta_\fM(s) = \sum_{n} \sigma(n) n^{-s} = \zeta(s) \zeta(s-1).
\end{equation}
So following Saito's approach is not necessary in this easy case. Also, the zeta function $\zeta(s)\zeta(s-1)$ for $\M_2^\ns(\ZZ)$ already appeared implicitly in the work of Connes and Marcolli (see e.g.\ \cite{ConnesMarcolli}), when they show that it is the partition function of their $\GL_2$-system.
\end{remark}

\begin{proposition}
Consider the big picture $\fP$ as a subgraph of $\fM$. Then its zeta function is given by
\begin{equation*}
\zeta_\fP(s) ~=~ \sum_{x \in \fP} \det(x)^{-s} ~=~ \frac{\zeta(s)\zeta(s-1)}{\zeta(2s)}.
\end{equation*}
\end{proposition}
\begin{proof}
In Proposition \ref{prop:bigpicture} we proved that $\mathfrak{P}$ is a fundamental domain for the action of $\NN_+^\times$ on $\fM$. So $\fM$ can be written as a disjoint union
\begin{equation}
\fM ~=~ \bigsqcup_{n \in \NN_+} n \cdot \fP
\end{equation}
The zeta function for $n \cdot \fP$ is given by $n^{-2s} \cdot \zeta_\fP(s)$, so we get
\begin{equation}
\zeta_\fM(s) = \sum_{n \in \NN_+} n^{-2s} \cdot \zeta_\fP(s)  = \zeta(2s)\zeta_\fP(s).
\end{equation}
The statement then follows from $\zeta_\fM(s) = \zeta(s)\zeta(s-1)$.
\end{proof}

\begin{remark}
We can also write the zeta function for $\fP$ as
\begin{equation}
\zeta_\fP(s) = \sum_{X \in \QQ_+\!\times\QQ/\ZZ} \det(\beta_X).
\end{equation}
with $\beta_X$ as in (\ref{eq:betaX}).
The analogous definition
\begin{equation}
\xi_\fP(s) = \sum_{X \in \QQ_+\!\times\QQ/\ZZ} \det(\alpha_X)
\end{equation}
with $\alpha_X$ as in (\ref{eq:alphaX}) is not well-defined, because for a fixed $n$ there are infinitely many $X \in \QQ_+ \times \QQ/\ZZ$ with $\det(\alpha_X) = n$.
\end{remark}

\section{Applications} \label{sec:applications}

In Section \ref{sec:topos}, we proved that the set of points for $\setswith{\M_2^\ns(\ZZ)}$ is given by
\begin{equation}
\GL_2(\ZZ) \backslash \M_2(\AA_f) \slash \GL_2(\QQ)
\end{equation}
and that, additionally, this double quotient classifies abelian groups $\ZZ^2 \subseteq A \subseteq \QQ^2$ up to isomorphism. Here $\AA_f = (\prod_p \ZZ_p) \otimes \QQ$ denotes the finite ad\`eles.

In this section we discuss some applications, with as underlying goal to determine to what extent this description is suitable for calculations.

\subsection{Relation to \texorpdfstring{$\ext^1(\QQ,\ZZ)$}{Ext(Q,Z)}} 

The Ext-group $\ext^1(\QQ,\ZZ)$ can be written as
\begin{equation} \label{eq:ext}
\ext^1(\QQ,\ZZ) \cong \AA_f / \QQ.
\end{equation}
For a proof using the long exact sequence we refer to Boardman's note \cite{Boardman}. We also refer to \cite{} (and the blogpost \cite{}) where an analogon for the full ring of adeles is discussed. In this subsection we provide an alternative proof of (\ref{eq:ext}) using Theorem \ref{thm:points}. From this approach we automatically get a criterion describing when two extension of $\QQ$ by $\ZZ$ are isomorphic as abelian groups (equivalent extensions are always isomorphic as abelian groups, but the converse does not hold).

We saw in the proof of Theorem \ref{thm:points} that for every subgroup $\ZZ^2 \subseteq A \subseteq \QQ^2$ there is an $x \in \GL_2(\ZZZ)\backslash\M_2(\ZZZ)$ such that
\begin{equation}
A_x = \varinjlim_{m \leq x} \ZZ^2
\end{equation}
where the filtered colimit is over the $m \leq x$ with $m \in \GL_2(\ZZ)\backslash\M_2^\ns(\ZZ)$ and where a transition map $\ZZ^2 \to \ZZ^2$ corresponding to $m \leq m'$ is given by $a$, with $a$ the matrix such that $m' = am$ (we assume that $m$, $m'$ are in Hermite normal form, in order to fix a matrix representative). Alternatively,
\begin{equation}
A_x = \{ (u,v) \in \QQ^2 ~:~ x \cdot (u,v) \in \ZZZ^2 \}
\end{equation}
where $(u,v)$ is seen as a column vector in $\AA_f^2$ on which $x$ acts by matrix multiplication.

We will focus on the subgroups $A_x$ such that the sequence
\begin{equation}
\begin{tikzcd}[row sep=tiny]
0 \ar[r]& \ZZ \ar[r]      & A_x  \ar[r]         & \QQ \ar[r] & 0  \\
        & n \ar[r,mapsto] & (n,0)               &            &    \\
        &                 & (u,v) \ar[r,mapsto] & v          &    \\
\end{tikzcd}
\end{equation}
is exact. In other words, we consider the subgroups $\ZZ^2 \subseteq A_x \subseteq \QQ^2$ with the properties
\begin{enumerate}[label=(E\arabic*),ref=(E\arabic*)]
\item $(u,0) \in A_x$ implies that $u \in \ZZ$; and \label{E1}
\item for all $v \in \QQ$ there is an $u \in \QQ$ with $(u,v) \in A_x$. \label{E2}
\end{enumerate}
By definition $A_x$ then determines an element $[A_x]\in\ext^1(\QQ,\ZZ)$ and it is easy to see that, conversely, every element of $\ext^1(\QQ,\ZZ)$ is the class of some $A_x$.

We now describe the elements $x \in \GL_2(\ZZZ)\backslash\M_2(\ZZZ)$ such that $A_x$ satisfies \ref{E1} and \ref{E2}. First we introduce the supernatural numbers as subset of the profinite integers.

\begin{definition} \label{def:supernatural}
A supernatural number (or Steinitz number) is a profinite integer $s \in \ZZZ$ such that for each prime $p$ its projection on $\ZZ_p$ (i.e.\ the $p$th component) is either $0$ or a power of $p$. The supernatural numbers will be denoted by $\SS$. They are a set of representatives for $\ZZZ^\times\backslash\ZZZ$.

With $p^\infty$ we denote the supernatural number such that the $p$th component is $0$ and such that the $q$th component is $1$ for each $q \neq p$.

For each natural number $n \in \NN$ we define $s(n)$ to be the supernatural number such that for each prime $p$ its projection on $\ZZ_p$ is $p^k$, where $p^k$ is the largest $p$th power dividing $n$.
\end{definition}
Note that $s(0)=0 = \prod_p p^\infty$ and $s(1) = 1$, but these are the only $n \in \NN$ for which $s(n)=n$. Our definition of the supernatural numbers agrees with the usual definition, apart from the fact that the supernatural numbers are usually defined abstractly as a monoid under multiplication (not as a subset of the profinite integers).

The supernatural numbers $\SS$ come into the picture when considering the Hermite normal form for $x \in \GL_2(\ZZZ)\backslash\M_2(\ZZZ)$. It is given by
\begin{equation}
x = \begin{pmatrix}
s & z \\
0 & s'
\end{pmatrix}
\end{equation}
with $s,s' \in \SS$ and $z \in \ZZZ$. Note that two matrices in Hermite normal form might describe the same element of $\GL_2(\ZZZ)\backslash\M_2(\ZZZ)$, for example
\begin{equation}
\begin{pmatrix}
0 & 1 \\
0 & 0
\end{pmatrix} ~\text{ and }~\begin{pmatrix}
0 & 0 \\
0 & 1
\end{pmatrix}.
\end{equation}
So what are the matrices
\begin{equation*}
x = \begin{pmatrix}
s & z \\
0 & s'
\end{pmatrix}
\end{equation*}
such that $A_x$ satisfies \ref{E1} and \ref{E2}? First we use that $(u,0) \in A_x$ if and only if
\begin{equation}
\begin{pmatrix}
s & z \\
0 & s'
\end{pmatrix} \begin{pmatrix}
u \\ 0
\end{pmatrix} = \begin{pmatrix}
su \\ 0
\end{pmatrix} \in \ZZZ^2,
\end{equation}
which is the case if and only if $su \in \ZZZ$. If $p \mid s$ then $(\tfrac{1}{p}, 0) \in A_x$, so it follows that $A_x$ satisfies \ref{E1} if and only if $s=1$. More generally, $(u,v) \in A_x$ if and only if
\begin{equation}
\begin{pmatrix}
s & z \\
0 & s'
\end{pmatrix} \begin{pmatrix}
u \\ v
\end{pmatrix} = \begin{pmatrix}
su + zv \\
s'v
\end{pmatrix} \in \ZZZ^2.
\end{equation}
Now it is easy to see that \ref{E1} and \ref{E2} hold if and only if $s = 1$ and $s' = 0$. So the matrices under consideration are of the form
\[
x = \begin{pmatrix}
1 & z \\
0 & 0 
\end{pmatrix}.
\]
As a group under multiplication, they can be identified with the additive group of profinite integers $\ZZZ$. Further, suppose that
\begin{equation}
x = \begin{pmatrix}
1 & z \\
0 & 0
\end{pmatrix} ~\text{ and }~ x' = \begin{pmatrix}
1 & z' \\
0 & 0
\end{pmatrix}
\end{equation}
determine an equivalent extension (i.e.\ the same element in $\ext^1(\QQ,\ZZ)$). Then $A_{x'} = g\cdot A_x$ for some $g \in \GL_2(\QQ)$ that preserves both the inclusion of $\ZZ$ and the projection on $\QQ$. We write
\[
g = \begin{pmatrix}
a & b \\
c & d
\end{pmatrix}.
\]
Then $g$ preserves $\ZZ$ if and only if
\begin{equation}
\begin{pmatrix}
a & b \\
c & d
\end{pmatrix} \begin{pmatrix}
1 \\ 0
\end{pmatrix} = \begin{pmatrix}
1 \\ 0
\end{pmatrix},
\end{equation}
in other words, if and only if $a=1$ and $c=0$. Moreover, $g$ preserves the projection on $\QQ$ if and only if
\begin{equation}
\begin{pmatrix}
0 & 1
\end{pmatrix}\begin{pmatrix}
a & b \\
c & d
\end{pmatrix} = \begin{pmatrix}
0 & 1
\end{pmatrix},
\end{equation}
in other words, if and only if $c=0$ and $d=1$. So $g$ is of the form
\[
g = \begin{pmatrix}
1 & b \\
0 & 1
\end{pmatrix}
\]
and we get
\begin{equation}
\begin{pmatrix}
1 & z' \\
0 & 0
\end{pmatrix} = \begin{pmatrix}
1 & z \\
0 & 0
\end{pmatrix} \begin{pmatrix}
1 & -b \\
0 & 1
\end{pmatrix} = \begin{pmatrix}
1 & z - b \\
0 & 0
\end{pmatrix}
\end{equation}
(note that $g \cdot A_x = A_{xg^{-1}}$). This shows
\begin{equation}
\ext^1(\QQ,\ZZ) ~=~ \ZZZ/\ZZ ~=~ \AA_f/\QQ.
\end{equation}

Even if $A_x$ and $A_{x'}$ define non-equivalent extensions, it is still possible that they are isomorphic as abelian groups. Any isomorphism $A_x \cong A_{x'}$ is given by conjugation by an element of $\GL_2(\QQ)$. So if
\begin{equation*}
x = \begin{pmatrix}
1 & z \\
0 & 0
\end{pmatrix} ~\text{ and }~ x' = \begin{pmatrix}
1 & z' \\
0 & 0
\end{pmatrix}
\end{equation*}
then there is a matrix $g = \begin{pmatrix}
a & b \\
c & d
\end{pmatrix} \in \GL_2(\QQ)$
such that
\begin{equation}
\begin{pmatrix}
1 & z \\
0 & 0
\end{pmatrix}
\begin{pmatrix}
a & b \\
c & d
\end{pmatrix} = \begin{pmatrix}
a + cz & b + dz \\
0 & 0
\end{pmatrix} ~\text{ and }~ \begin{pmatrix}
1 & z' \\
0 & 0
\end{pmatrix}
\end{equation}
define the same element in $\GL_2(\ZZZ)\backslash\M_2(\AA_f)$. This is the case if and only if
\begin{equation}
a + cz \in \ZZZ^\times~\text{ and }~z' = \frac{b + dz}{a + cz}.
\end{equation}

\begin{proposition} \label{prop:ext-isomorphic}
Consider the partially defined right action of $\PGL_2(\QQ)$ on 
\begin{equation*}
\ext^1(\QQ,\ZZ) = \AA_f/\QQ
\end{equation*}
given by
\[
z \cdot \begin{pmatrix}
a & b \\
c & d
\end{pmatrix} ~=~ \frac{b + dz}{a + cz}
\]
whenever $a + cz \in \AA_f^\times$. Then two extensions $A,A' \in \ext^1(\QQ,\ZZ)$ are isomorphic as abelian groups
if and only if they are in the same $\PGL_2(\QQ)$-orbit.
\end{proposition}
\begin{proof}
This follows from the above discussion. Note that for $a + cz \in \AA_f^\times$ we can assume that $a + cz \in \ZZZ^\times$, because $\begin{pmatrix}
a & b \\
c & d
\end{pmatrix}$ is only defined up to a scalar.
\end{proof}

\subsection{Computations}

We now try to determine whether some extensions in $$\ext^1(\QQ,\ZZ) = \AA_f/\QQ$$ are isomorphic as abelian groups. This will reveal some advantages and limitations of the description from Proposition \ref{prop:ext-isomorphic}.

Recall from Definition \ref{def:supernatural} the constuction of supernatural numbers as subset of $\ZZZ$, and the specific supernatural numbers $s(n)$ with $p$th component given by the largest $p$th power dividing $n$.

We first consider the set
\begin{equation}
N = \{ s(n) : n \in \NN  \} ~\subseteq~ \AA_f/\QQ  = \ext^1(\QQ,\ZZ);
\end{equation}
when do $s(n)$ and $s(m)$ define isomorphic abelian groups? In other words, when can we find some
\begin{equation}
\begin{pmatrix}
a & b \\
c & d
\end{pmatrix} \in \PGL_2(\QQ) ~\text{ such that }~ s(m) = \frac{b + d\, s(n)}{a + c\, s(n)}~?
\end{equation}
In the following, we use the notation $s(m)\sim s(n)$ when the above holds.

\begin{proposition} \ 
\begin{enumerate}
\item if $s(n) \sim s(0)$ then $n \in \{0,1\}$;
\item if $s(n) \sim s(m)$ for $n,m \notin \{0,1\}$, then $n$ and $m$ have the same prime divisors;
\item $s(p^k) \sim s(p^u)$ for all primes $p$ and integers $k,u \geq 1$;
\item $s(p^kq^r) \sim s(p^u q^v)$ for all primes $p,q$ and integers $k,r,u,v \geq 1$.
\end{enumerate}
\end{proposition} \label{prop:same-prime-divisors}
\begin{proof} \ 
\begin{enumerate}
\item Note that $s(0) = 0$, so if $s(n) \sim s(0)$ then $s(n) \in \QQ$. It is clear that $n = 0,1$ are possible. Conversely, if $n \neq 0$ then the $p$th component of $s(n)$ is $1$ for almost all primes $p$. Together with $s(n) \in \QQ$ this shows $s(n) = 1$, so $n = 1$.
\item Suppose that $s(n) \sim s(m)$, more precisely
\[
s(m) = \frac{b + d\,s(n)}{a + c\,s(n)}.
\]
Then $a\,s(m) + c\,s(nm) = b + d\,s(n)$. Note that for almost all primes $p$, the $p$th components of both $s(n)$ and $s(m)$ are $1$. By looking at such a component we see that $a + c = b + d$ and by rescaling we can assume $a + c = 1 = b + d$. Now suppose that there is a prime $q$ such that the $q$th components of $s(n)$ and $s(m)$ are $1$ resp.\ $q^v$. Then $q^v = aq^v + cq^v = b + d = 1$.
\item Note that $s(p^u) = \tfrac{p^k-p^u}{p^k-1} + \tfrac{p^u-1}{p^k-1}\,s(p^k)$; this can be checked componentwise.
\item It is enough to show that there is a solution to the system of equations
\begin{equation*}
\begin{cases}
b + d = 1 \\
a + c = 1 \\
b + dp^k = ap^u + cp^{k+u} \\
b + dq^r = aq^v + cq^{r+v}
\end{cases}.
\end{equation*}
Then $s(p^uq^v) = \tfrac{b + d\,s(p^kq^r)}{a + c\,s(p^kq^r)}$ and moreover $a + c\,s(p^kq^r) \in \AA_f^\times$ (indeed, if the $p$th component of $a + c\,s(p^kq^r)$ would be zero, then $a + cp^k = 0 = b + dp^k$; together with $a + c = 1 = b+d$ we find $(a,c)=(b,d)$ but this contradicts $b + dq^r = aq^v + c q^{r+v}$).
The system of equations has a solution because
\[
\begin{vmatrix}
 1 & 1 & 0 & 0 \\
 0 & 0 & 1 & 1 \\
 1 & p^k & p^u & p^{k+u} \\
 1 & q^r & q^v & q^{r+v} \\
\end{vmatrix} = (p^k-1)(q^r-1)(p^u-q^v) \neq 0.
\]
\end{enumerate}
\end{proof}

For more general extensions, the situation becomes a lot more complicated. Recall that the Goormaghtigh conjecture\footnote{Named after the Belgian engineer/mathematician Ren\'e Goormaghtigh.} states that
the only natural number solutions to
\begin{equation}
\frac{x^n-1}{x-1} = \frac{y^m-1}{y-1}
\end{equation}
are $(x,y,n,m) = (2,5,5,3)$ and $(x,y,n,m) = (2,90,13,2)$. The conjecture is still open at the time of writing.

\begin{proposition}[Relation to Goormaghtigh conjecture]
We have
\[
s(2^4\cdot 5^2)l^\infty \sim s(2^5\cdot 5^3)l^\infty
\]
for all primes $l$. Any other solution $(p,q,l,k,r)$, $p \leq q$ of
\[
s(p^k q^r) l^\infty \sim s(p^{k+1}q^{r+1}) l^\infty
\]
gives a counterexample to Goormaghtigh conjecture.
\end{proposition}
\begin{proof}
We can check componentwise that
\begin{equation}
\frac{30\,s(2^4\cdot 5^2) l^\infty}{31-s(2^4\cdot 5^2)l^\infty} = s(2^5\cdot 5^3).
\end{equation}
Further, if 
\begin{equation}
\frac{b + d \, s(p^k q^r)l^\infty }{a + c \, s(p^k q^r)l^\infty} = s(p^{k+1}q^{r+1})l^\infty,
\end{equation}
then we can assume $b + d = 1 = a+c$ like in the proof of (2), and by looking at the components we get
\begin{equation}
\begin{cases}
b + d = 1 \\
a + c = 1 \\
b + dp^k = ap^{k+1} + cp^{2k+1} \\
b + dq^r = aq^{r+1} + cq^{2r+1} \\
b = 0
\end{cases}.
\end{equation}
From this we find
\begin{equation}
\begin{cases}
1 = a + c \\
1 = ap + cp^{k+1} \\
1 = aq + cq^{k+1}
\end{cases}
\end{equation}
so 
\begin{equation}
a = -\tfrac{p^{k+1}-1}{p-1}\, c = - \tfrac{q^{r+1}-1}{q-1}\, c
\end{equation}
but this means that $(p,q,k+1,r+1)$ is a counterexample to Goormaghtigh conjecture, except when $p^{k+1} = 2^5$ and $q^{r+1} = 5^3$.
\end{proof}

\section*{Acknowledgements}

I would like to thank Lieven Le Bruyn for coming up with the problem of describing toposes associated to $2\times2$ integer matrices, and for helping me understand Conway's big picture and the work of Connes and Consani.

Also thanks to Yinhuo Zhang for leading me to the relation with Goormaghtigh conjecture (Section \ref{sec:applications}), to Julia Ramos Gonz\'alez for discussions about categories of $R$-sets versus categories of $R$-modules and to Pieter Belmans for helping me out regarding Hasse-Weil zeta functions.

\end{document}